\documentclass[a4paper,oneside,11pt]{article}

\usepackage{amsmath,amssymb,amsthm}
\usepackage[hidelinks]{hyperref}
\usepackage{yfonts}
\usepackage{enumerate}
\usepackage{color}
\usepackage{mathtools}
\usepackage{thmtools,thm-restate}
\usepackage{tikz}
\usepackage[justification=centering]{caption}
\usepackage{ wasysym } %% for \clock \diameter

\usepackage[ruled]{algorithm2e}
\SetArgSty{textup}

\usepackage{geometry}
\renewcommand{\Pr}{\mathbb P}

\newcommand{\OO}{\mathcal{O}}

\newcommand{\myimg}[2]{\IfFileExists{#1}{\includegraphics[scale=#2]{#1}   }{  Image not found  \errmessage{image #1 is missing}}}
\renewcommand{\epsilon}{\varepsilon}

\newtheorem{theorem}{Theorem}[section]
\newtheorem{lemma}[theorem]{Lemma}
\newtheorem{proposition}[theorem]{Proposition}

\newtheorem{definition}[theorem]{Definition}

\newtheorem{observation}[theorem]{Observation}

\newcommand\spmm{\operatorname{spmm}}
\newcommand\bmm{\operatorname{bmm}}
\newcommand\bwmm{\operatorname{bwmm}}

\title{Mastermind with a Linear Number of Queries}

\newcommand*\samethanks[1][\value{footnote}]{\footnotemark[#1]}
\author{ 
Anders Martinsson%
\thanks{Department of Computer Science, ETH Zurich, Switzerland.
        Email: \{anders.martinsson$\vert$sup\}@inf.ethz.ch} 
\and
Pascal Su%
\samethanks[1]%
\ \thanks{author was supported by grant no. 200021 169242 of the Swiss National Science Foundation}
}

\date{\today}

\begin{document}
  \maketitle

\begin{abstract}
Since the 1960s Mastermind has been studied for the combinatorial and information-theoretical interest the game has to offer. Many results have been discovered starting with Erd\H{o}s and R\'enyi determining the optimal number of queries needed for two colors. For $k$ colors and $n$ positions, Chv\'atal found asymptotically optimal bounds when $k \le n^{1-\epsilon}$. Following a sequence of gradual improvements for $k\geq n$ colors, the central open question is to resolve the gap between $\Omega(n)$ and $\mathcal{O}(n\log \log n)$ for $k=n$.

In this paper, we resolve this gap by presenting the first algorithm for solving $k=n$ Mastermind with a linear number of queries. As a consequence, we are able to determine the query complexity of Mastermind for any parameters $k$ and $n$.

\end{abstract}

%!TEX root = mastermind.tex

\section{Introduction}

Mastermind is a famous code-breaking board game for two players. One player, codemaker, creates a hidden codeword consisting of a sequence of four colors. The goal of the second player, the codebreaker, is to determine this codeword in as few guesses as possible. After each guess, codemaker provides a certain number of black and white pegs indicating how close the guess is to the real codeword. The game is over when codebreaker has made a guess identical to the hidden string.

The board game version was first released in 1971, though the idea of the game is older, and variations of this game have been played earlier under other names, such as the pen-and-paper based games of Bulls and Cows, and Jotto. The game has been played on TV as a game show in multiple countries under the name of Lingo. Recently, a similar web-based game has gained much attention under the name of Wordle.

Guessing games such as Mastermind have gained much attention in the scientific community. This is in part due to their popularity as recreational games, but importantly also as natural problems in the intersection of information theory and algorithms. In particular, it is not too hard to see that two-color Mastermind is equivalent to coin weighing with a spring scale. This problem was first introduced in 1960 by Shapiro and Fine \cite{SFproblemE1399}. In subsequent years, a number of different approaches have been devised which solve this problem up to a constant factor of the information-theoretic lower bound.

The general $k$ color $n$ slot Mastermind first appeared in the scientific literature in 1983 in a paper by Chv\'atal \cite{chvatal1983mastermind}. By extending ideas of Erd\H{o}s and R\'enyi \cite{ERcoin} from coin-weighing he showed that the information-theoretic lower bound is sharp up to a constant factor for $k\leq n^{1-\varepsilon}$ for any fixed $\varepsilon>0$.

Surprisingly, for a larger number of colors, the number of guesses needed to reconstruct the codeword has remained unknown. In particular, for $k=n$, the best known upper bound remained for a long time $\OO(n \log n)$ as shown by Chv\'atal, with only constant factor improvements given in \cite{CCH96}, \cite{Goo09} and \cite{JP11}. Only quite recently, this bound was improved to $\OO(n \log \log n)$ in an article by Doerr, Doerr, Sp\"ohel, and Thomas \cite{DDST16} published in Journal of the ACM in 2016. At the same time, no significant improvement on the information-theoretic lower bound of $n$ queries has been obtained.

The challenge of finding short solutions to Mastermind in the setting of $k=n$ colors and positions can be thought of as a bootstrapping paradox. On the one hand, a query could in principle return anything between $0$ and $n$ black pegs, awarding codebreaker potentially with as many as $\log(n+1)$ bits of information about the codeword. As it takes $\log(n^n)=n\log n$ bits of information to encode an arbitrary codeword, it follows that any strategy needs $\Omega(n\log n/\log(n+1))=\Omega(n)$ queries. On the other hand, given no prior information about the codeword, codebreaker can only expect to get a constant number of correct guesses, meaning that the query only awards codebreaker with $\OO(1)$ bits of information\footnote{For simplicity, we here assume that codebreaker only receives the number of black pegs for each query. The situation is similar if also white pegs are provided, but it takes additional arguments to see that not too much additional information can come from the white pegs.}. Morally, one expects queries with higher information content to gradually become available as codebreaker gains information about the codeword, but making sense of this formally has turned out to be difficult problem. In particular, should one expect the entropy lower bound of $\Omega(n)$ to be the truth, or could it be that this bound is unattainable as it takes too long to reach the point where queries give the full $\Theta(\log n)$ bits of information?

In this paper, we resolve this problem after almost 40 years by showing how the $n$ color $n$ slot Mastermind can be solved with $\OO(n)$ guesses with high probability, matching the information-theoretic lower bound up to a constant factor. By combining this with a result by Doerr, Doerr, Sp\"ohel, and Thomas \cite{DDST16}, we determine asymptotically the optimal number of guesses for all $k$ and $n$.

%!TEX root = mastermind.tex

\subsection{Game of Mastermind}

We define a game of Mastermind with $k \geq 1$ \emph{colors} and $n \geq 1$ \emph{positions} as a two-player game played as follows. One player, the codemaker, initially chooses a hidden codeword $c = (c_1, ...,c_n)$ in $[k]^n$. The other player, the codebreaker, is then tasked with determining the hidden codeword by submitting a sequence of queries of the form $q = (q_1, ...,q_n) \in [k]^n$. For each query, the codemaker must directly respond with information on how well the query matches the codeword. The codebreaker may use this information to adapt subsequent queries. The precise information given depends on which variation of Mastermind is played, as will be specified below. The game is over as soon as codebreaker makes a query such that % $b_c(q)=n$, that is,
$q=c$. The goal of the codebreaker is to make the game ends after as few queries as possible.

For each pair of a codeword $c$ and a query $q$, we associate two integers called the number of \emph{black pegs} and \emph{white pegs} respectively.  The number of black pegs,
\[b(q)=b_c(q) := |\{i\in [n] : q_i=c_i\}|,\]
is the number of positions in which the codeword matches the query string. The number of white pegs is often referred to as the number of correctly guessed colors that do not have the correct position. More precisely, the number of white pegs,
\[w(q)=w_c(q):= \max_{\sigma \in S_n} |\{i\in [n] : q_{\sigma(i)}=c_i\}| - b_c(q),\]
is the number of additional correct positions one can maximally obtain by permuting the entries of $q$.

In this paper, we will consider two versions of Mastermind. First, in \emph{black-peg} Mastermind the codemaker gives only the number of black pegs as an answer to every query. Second, in \emph{black-white-peg} Mastermind the codemaker gives both the number of black and the number of white pegs as answers to every query.

\subsection{Our results}
Our contribution lies in resolving the black-peg Mastermind game for $k = n$ where we have as many possible colors as positions. 

\begin{restatable}{theorem}{main}
\label{thm:main}There exists a randomized algorithm that solves black-peg Mastermind with $n$ colors and $n$ positions with high probability using $\OO(n)$ queries. Moreover, the runtime of the algorithm is polynomial in $n$.
\end{restatable}

The above result is best possible up to a constant factor. This can be seen by observing that there are at most $n+1$ possible answers to any query. Hence codebreaker gains at most $\log_2(n+1)$ bits of information from a query. As $\log_2(n^n)=n\log_2 n$ bits are required to uniquely determine the codeword, any strategy needs at least $\Omega(n)$ queries. Moreover, if we additionally assume that the codeword can only be a permutation of the colors, that is each color must appear exactly once, then we can solve it deterministically with $\OO(n)$ queries.

Combining our results with earlier results by Doerr, Doerr, Sp\"ohel, and Thomas we are able to resolve the randomized query complexity of Mastermind in the full parameter range, thus finally resolving this problem after almost 40 years.  We define the randomized query complexity as the minimum (over all strategies) maximum (over all codewords) expected number of queries needed to win the game.

\begin{restatable}{theorem}{general}
\label{thm:general}
For $k$ colors and $n$ positions, the randomized query complexity of Mastermind is 
\[\Theta(n \log k / \log n + k/n),\] 
if codebreaker receives both black-peg and white-peg information for each query, and 
\[\Theta(n \log k / \log n + k),\] 
if codebreaker only receives black-peg information.
\end{restatable}

We believe the same result holds true for the deterministic query complexity, but we will not attempt to prove it here. 

\subsection{Related work}

The study of two-color Mastermind dates back to an American Mathematical Monthly post in 1960 by Shapiro and Fine \cite{SFproblemE1399}: ``\emph{Counterfeit coins weigh 9 grams and genuine coins all weigh 10 grams. One is given $n$ coins of unknown composition, and an accurate scale (not a balance). How many weighings are needed to isolate the counterfeit coins?}''

It can be observed, already for $n=4$, that fewer than $n$ weighings are required. The authors consequently conjecture that $o(n)$ weighings suffice for large $n$. Indeed, the entropy lower bound states that at least $n/\log_2(n+1)$ weighings are necessary. In the subsequent years, many techniques were independently discovered that attain this bound within a constant factor \cite{ERcoin, cantor1966determination, lindstrom1964combinatory, lindstrom1965combinatorial}, see also \cite{ERcoin} for further early works. Erd\H{o}s and R\'enyi \cite{ERcoin} showed that a sequence of $(2+o(1))n/\log_2 n$ random weighings would uniquely identify the counterfeit coins with high probability, and by the probabilistic method, there is a deterministic sequence of $\OO(n/\log n)$ weighings that identify any set of counterfeit coins.

Cantor and Mills \cite{cantor1966determination} proposed a recursive solution to this problem. Here it is natural to consider \emph{signed} coin weighings, where, for each coin on the scale, we may choose whether it contributes with its weight or minus its weight. We call a $\{-1, 0, 1\}$ valued matrix $A$ an \emph{identification matrix} if any binary vector $x$ of compatible length to $A$ can be uniquely determined by the values of $Ax$. So for example $\begin{pmatrix} 1 & 0\\ 0 & 1\end{pmatrix}$ is an identification matrix for binary vectors of length $2$. It is not too hard to show that if $A$ is an identification matrix, then so is 
\begin{equation}\label{eq:cantormills}\begin{pmatrix} A & A & I\\ A & -A & 0\end{pmatrix}.
\end{equation}
By putting $A_0=(1)$ and recursing this formula, we obtain an identification matrix $A_k$ with $2^k$ rows and $(k+2)2^{k-1}$ columns. Thus, we can identify which out of $n=(k+2)2^{k-1}$ coins are counterfeit by using $2^k \sim 2 n/\log_2 n$ \emph{signed} weighings, or, by weighing the $+1$s and $-1$s separately, using $\sim 4 n/\log_2 n$ (unsigned) weighings. Using a more careful analysis, the authors show that $\sim 2 n/\log_2 n$ weighings suffice.

It was shown in \cite{ERcoin} that $2n/ \log_2 n$ is best possible, up to lower-order terms, for \emph{non-adaptive} strategies. It is a central open problem to determine the optimal constant for general strategies, but it is currently not known whether adaptiveness can be used to get a leading term improvement.

Knuth \cite{Knuth77} studied optimal strategies for the commercially available version of Mastermind, consisting of four positions and six colors. He showed that the optimal deterministic strategy needs $5$ guesses in the worst case. In the randomized setting, it was shown by Koyama \cite{Koyama93} that optimal strategy needs in expectation $5625/1296 = 4.34\dots$ guesses for the worst-case distribution of codewords.

The generalization of Mastermind to $k$ colors and $n$ positions first appeared in the scientific literature in 1983 in a paper by Chv\'atal \cite{chvatal1983mastermind}, who attributed the idea to Pierre Duchet. Here the entropy lower bound states that $\Omega(n \log k/ \log n)$ guesses are necessary. For $k\leq n^{1-\varepsilon}$, Chv\'atal showed that a simple random guessing strategy uniquely determines the codeword within a constant factor of the entropy bound.

For larger $k$, less has been known. For $k$ between $n$ and $n^2$, Chv\'atal showed that $2n\log_2 k + 4n$ guesses suffice. For any $k\geq n$, this was improved to $2n \log_2 n +2n+\lceil k/n\rceil + 2$ by Chen, Cunha, and Homer \cite{CCH96}, further to $n\lceil \log_2 k\rceil + \lceil (2-1/k)n\rceil + k$ by Goodrich \cite{Goo09}, and again to $n\lceil \log_2 n\rceil -n+k+1$ by J\"ager and Peczarski \cite{JP11}. As a comparison, we note that if $k=k(n)$ is polynomial in $n$, then the entropy lower bound is simply $\Omega(n)$. This gap is a very natural one as Doerr, Doerr, Sp\"ohel, and Thomas \cite{DDST16} showed in a relatively recent paper that if one uses a non-adaptive strategy, there is in fact a lower bound of $\Omega(n \log n)$ when $k = n$. In the same paper, they also use an adaptive strategy to significantly narrow this gap, showing that $\OO(n \log \log n)$ guesses suffice for $k=n$. Moreover, they present a randomized reduction from black-white-peg Mastermind to black-peg Mastermind which shows that
$$ \operatorname{bwmm}(n, k) = \Theta( k/n + \operatorname{bmm}(n, n)),$$
for any $k\geq n$ where $\operatorname{bwmm}(n, k)$ denotes the randomized query complexity for black-white-peg Mastermind with $k$ colors and $n$ positions, and where $\operatorname{bmm}(n, n)$ denotes the randomized query complexity for black-peg Mastermind with $n$ colors and positions. As a consequence, they concluded that $\operatorname{bwmm}(n, k)=\OO(n \log \log n + k/n)$ for all $k$ and $n$.

Stuckman and Zhang \cite{SZ06} showed that it is NP-hard to determine whether a sequence of guesses with black and white peg answers is consistent with any codeword. The analogous result was shown by Goodrich \cite{Goo09} assuming only black-peg answers are given. It was shown by Viglietta \cite{Vig11} that both of these results hold even for $k=2$.

Variations of Mastermind have furthermore been proposed to model problems in security, such as revealing someone's identity by making queries to a genomic database \cite{Googene}, and API-based attacks to determine bank PINs  \cite{FL10}.
%!TEX root = mastermind.tex

\section{Proof outline}
\label{sec:proofoutline}

The proof of Theorem \ref{thm:main} can be broken up in two main steps. In the first step, contained in Section \ref{sec:manycol}, we reduce the traditional form of Mastermind to what we call \emph{signed permutation Mastermind}. This can be described as the variation of Mastermind on $n$ positions and colors where
\begin{enumerate}
\item codemaker is restricted to choosing the codeword $c$ to be a permutation of $[n]$, and 
\item codebreaker can make signed queries $q\in\{-n, \dots n\},$ where, after each query, codemaker must respond with the value
$$\hat{b}(q) = \hat{b}_c(q) := |\{i \in [n] | c_i = q_i \}| -  |\{i \in [n] | c_i = - q_i \}| .$$
\end{enumerate}
In other words, codebreaker can decide, for each position in a query, whether a correct guess should count as a $+1$ or as a $-1$. Note also that codebreaker can choose to leave an entry in a query ``blank'' by giving it the value $0$, in which case it will never contribute to the returned value.

This is all done in preparation for the second step, contained in Section \ref{sec:main}, where we take a more constructive approach and recursively as well as deterministically provide a set of adaptive queries that give the answer to the signed permutation Mastermind problem.

Our approach to determining the codeword in this version of Mastermind can be described as resolving $\OO(n \log n)$ subtasks of the form ``Given that a color $x$ is present in an interval $I$, determine whether $x$ is present in the left or right half of $I$''. In other words, we attempt a binary search for each color. We can make such a task part of a query by putting $q_i=x$ for all indices $i$ in the left half of $I$, and putting $q_i=0$ in all indices $i$ in the right half of $I$. Then this will contribute a with $+1$ to the answer of the query if color $x$ is in the left half of $I$, and $0$ and if $x$ is in the right half of $I$. Clearly, a single task can be resolved after a single query, however the challenge is to find a way to resolve these tasks by performing only $\OO(n)$ queries.

Similar to the algorithm by Doerr, Doerr, Sp\"ohel, and Thomas \cite{DDST16}, we will base our solution on coin-weighing schemes. Here, we think of the $\OO(n \log n)$ tasks described above as our coins, where a coin has the value $1$ if the color in the corresponding task is present in the left half of its interval, and $0$ otherwise. Weighing a collection of coins together corresponds to making one query where each of the corresponding tasks is encoded as described above. Note that if there were no further restrictions on how tasks could be queried together, then any of the classical solutions to coin-weighing would let us resolve the $\OO(n \log n)$ tasks in $\OO( n \log n / \log( n \log n)) = \OO(n)$ queries, which is the conclusion we want, but under far too weak assumptions. The classical coin-weighing problem assumes that \emph{any} collection of coins can be weighed together. This is far from true in the problem at hand. First, we cannot encode two tasks into the same query if their corresponding intervals overlap. Second, tasks depend on each other in the sense that certain tasks are only available for querying after another task is resolved. For instance, we cannot query in which quarter of the codeword the color red is present until we have first determined in which half of the codeword it is in. This is of particular concern for tasks corresponding to big intervals as, on the one hand, they, intuitively, have little potential to be resolved in parallel, and on the other hand, these are the only tasks that are available initially.

One natural approach to circumvent this is to resolve the tasks ordered by the interval size from large to small, layer by layer. That is, we first determine which half each color is present in by querying colors one at a time. We then determine which quarter each color is present in two at a time, and so on. That is, in the $i$th layer we query $2^i$ colors at a time. Using optimal coin-weighing schemes, this can be done in $\OO( (n/2^i) \cdot 2^i/\log(2^i) )=\OO(n/i)$ queries. Alas, this is too slow as summing this over all layers $i=0, \dots, \log_2 n$ gives a total of $\OO(n\log \log n)$ queries. In order to speed this up to $\OO(n)$ queries, we need to find a novel take on coin-weighing schemes where intervals of all different sizes are queried together in one big phase, which respects the dependencies between tasks.

In fact, our solution to this problem is surprisingly elegant. We start by performing a procedure we call \emph{preprocess}. The aim of which is to use $\OO(n)$ queries (so far without any information-theoretic speedup) in order to determine the positions of colors sufficiently well so that, after this point, it is possible to query $n^{\Omega(1)}$ colors together. We then perform a procedure we call \emph{solve} that employs the parallelism unlocked by \emph{preprocess} to determine the codeword in an additional $\OO(n)$ queries. This procedure is recursively constructed in a manner similar to divide and conquer algorithms, and using a technique similar to the coin-weighing scheme by Cantor-Mills \cite{cantor1966determination} to achieve the information-theoretic speedup.

Section \ref{sec:general} is then dedicated to the general case where the number of colors k and number of positions n may differ Mastermind. We prove Theorem \ref{thm:general}. This can be seen as a direct consequence of Theorem \ref{thm:main} and applying previous results by Doerr, Doerr, Sp\"ohel, and Thomas \cite{DDST16}.
%!TEX root = mastermind.tex

\section{Signed Permutation Mastermind}
\label{sec:manycol}
In the following section, we show how to reduce the traditional form of Mastermind to \emph{signed permutation Mastermind}, as defined in Section \ref{sec:proofoutline}.

\subsection{Finding a zero query}
We first show how one can simulate ``blank'' guesses in Mastermind.  To achieve this, it suffices to find a query $z$ such that $b(z)=0$, which can be done as follows.

\begin{lemma}\label{lem:zero} 
For any $k\geq 2$, it is possible to find a string $q$ with $b_c(q)=0$ using at most $n+1$ queries.
\end{lemma}

\begin{proof}
Query the string of all $1$s, $t^{(0)}$, as well as the strings $t^{(i)}$ for all $i \in \{1, .. , n\}$ where $t^{(i)}$ is the string of all ones except at the $i$th position it has a $2$. Now if $b_c(t^{(i)}) = b_c(t^{(0)}) -1$, then $c_i\neq 2$, otherwise $c_i\neq 1$ and we have at every position a color that is incorrect, so we have a string $z$ which satisfies $b_c(z)=0$.
\end{proof}

Note that if $k=n$ and we are content with a randomized search then we can find an all-zero string by choosing queries $z\in [n]^n$ uniformly at random until a query is obtained with $b_c(z)=0$. As the success probability of one iteration $(1-1/n)^n\geq 1/4$, this takes on average $4$ guesses.

\subsection{Finding pairwise elementwise distinct one queries}
\label{subsec:perm}

We say that a set of queries $f^{(1)}, f^{(2)}, \dots$ are pairwise elementwise distinct if $f^{(s)}_i\neq f^{(t)}_i$ for all $i \in [n]$ and all $s\neq t$. 

The second part of the reduction is to show how codebreaker can transform the problem to the setting where the codeword is a permutation of $[n]$. It turns out codebreaker can achieve this by first finding $n$ pairwise elementwise distinct queries $f^{(1)}, \dots f^{(n)}$ such that $b_c(f^{(i)})=1$ for all $i$. This can be done with high probability using $\OO(n)$ random queries in the following fashion. 

This argument is due to Angelika Steger (from personal communication).

\begin{algorithm}[t]
\caption{{$ Permutation $}} \label{alg:permutation}
\KwOut{$n$ pairwise elementwise distinct strings $f^{(1)}, \dots , f^{(n)}$ such that each string contains exactly one correct position}
\For{ $i \in [n]$ }{
$S_i^{(1)} = [n]$\;
}
t = 1\;
\While{$t \le n$}{
	$X \leftarrow $ Choose a random color for each position $i$ uniformly from $S_i^{(t)}$ for each $i\in [n]$ \;
	\If{$b(X) = 1$}{
		$f^{(t)} \leftarrow X$ \;
		\For{$i \in [n]$ }{
		$S_i^{(t+1)} = S_i^{(t)} \setminus \{X_i\}$ \;
		}
		$t = t+1$;
	}
}
\Return $f^{(1)}, \dots , f^{(n)}$
\end{algorithm}

\begin{lemma}\label{lem:disj}
Algorithm \ref{alg:permutation} will, with high probability, need at most $\OO(n)$ many queries to identify pairwise elementwise distinct query strings $f^{(1)}, \dots, f^{(n)}$ of length $n$ such that $b_c(f^{(i)})=1$ for all $i\in[n]$. 
\end{lemma}

\begin{proof}
The finding of these strings is done by random queries. Let $S^{(1)} = [n]^n$ start out to be the entire space of possible queries. Sample uniform random queries $X$ from $S^{(1)}$ until one of them gives $b(X) = 1$. Set this query to be $f^{(1)}$. Now set aside all colors at the corresponding positions and keep querying. That is $S^{(2)} = [n]\backslash \{f^{(1)}_1\} \times ... \times [n]\backslash \{f^{(1)}_n\}$ and sample again random queries $X$ from $S^{(2)}$ until one of them gives $b(X)= 1$. In this way set aside $f^{(i)}$ which was received by querying randomly from $S^{(i)} = [n]\backslash \{f^{(1)}_1,..,f^{(i-1)}_1\} \times ... \times [n]\backslash \{f^{(1)}_n,..,f^{(i-1)}_n\}$.

We analyze how many queries this takes. The set $S^{(i)}$ has $n-i+1$ many possible colors at every position and also $n-i+1$ many positions at which there is still a correct color available. So for each position where there still is an available correct color, there is a chance of $1/(n-i+1)$ that this is guessed correctly in $X$, independently of every other position. So the probability that $b(X) = 1$ for a random query is 
\begin{equation} 
\Pr[b(X) = 1] = (n-i+1) \cdot \frac{1}{n-i+1}  \left(1 - \frac{1}{n-i+1}\right)^{n-i} \ge e ^{-1}. \label{eq:coupon}
\end{equation}

Let $Y_t$ be the number of queries it takes to find the $t$th string to set aside. Then $Y_t$ is geometrically distributed and the total time is the sum of all $Y_t$, $t\in [n]$, which is an independent sum of geometrically distributed random variables with success probabilities as in (\ref{eq:coupon}), so expected at most $e^{-1}$. Applying the Chebychev inequality gives us that, w.h.p. we can find $f^{(1)}$ to $f^{(n)}$ with $\OO(n)$ many queries.
\end{proof}

\subsection{Reduction to Signed Permutation Mastermind}
\label{subsec:reduction}

The previous subsections set the stage for the following lemma, which shows the dependency between the signed permutation Mastermind and the black-peg Mastermind Problem. With only $\OO(n)$ additional queries that result from Lemma~\ref{lem:disj} it is possible to solve black-peg Mastermind assuming we can solve signed permutation Mastermind with the same number of queries up to a constant factor. 
We denote by $\spmm(n)$ the minimum number of guesses needed by any deterministic strategy to solve signed permutation Mastermind.

\begin{lemma}\label{lem:perm}
Black-peg Mastermind with $n$ colors and slots can be solved in $\OO(n + \spmm(n))$ guesses with high probability.
\end{lemma}

\begin{proof}
First apply the algorithms from Lemmas \ref{lem:zero} and \ref{lem:disj} to obtain the corresponding queries $z$ and $f^{(1)}, \dots, f^{(n)}$. Given this, run any optimal solution to signed permutation Mastermind. Whenever this solution wants to perform a query $q\in \{-n, \dots, n \}^n$, we construct queries $q^+, q^- \in [n]^n$ according to
$$q^+_i = \begin{cases} f^{(q_i)}_i &\text{ if }q_i>0\\ z_i &\text{otherwise,}\end{cases}$$
and
$$q^-_i = \begin{cases} f^{(-q_i)}_i &\text{ if }q_i<0\\ z_i &\text{otherwise.}\end{cases}$$
We consequently query $b_c(q^+)$ and $b_c(q^-)$ and return the value $b_c(q^+)-b_c(q^-)$ as the answer to query $q$.

We want to show that the answers given by this procedure are equal to $\hat{b}_{c'}(q)$ for some permutation $c'$ of $[n]$ not depending on $q$. Let $\varphi : [n]^n \rightarrow [n]^n $ be the map given by $\varphi(x)_i = f^{(x_i)}_i$. As $f^{(1)}_i, \dots, f^{(n)}_i$ are all distinct values between $1$ and $n$, by construction, it follows that $\varphi$ acts as a bijection on each position of the input. In fact, we claim that 
\begin{equation}\label{eq:toperm}
\hat{b}_{\varphi^{-1}(c)}(q)=b_c(q^+)-b_c(q^-) \qquad\forall q\in\{-n, \dots, n\}^n.
\end{equation}
In other words, any signed query $q$ made in this setting will be answered as if the codeword is $c':=\varphi^{-1}(c).$ Observe that \eqref{eq:toperm} implies that $c'$ is a permutation as $\hat{b}_{c'}(i\dots i)$ is, by the definition of $\hat{b}_{c'}(\cdot)$, equal to the number of occurrences of color $i$ in $c'$, and by definitions of $b_c(\cdot), q^+$ and $q-$ equal to $b_c(f^{(i)})-b_c(z)=1-0=1$, so each color appears exactly once in $c'$.

To see that \eqref{eq:toperm} holds, let $q\in \{-n, \dots, n\}^n$ let $q^+, q^-$ be as above, and consider the contributions to $\hat{b}_{c'}(q)$ and $b_c(q^+)-b_c(q^-)$ respectively from the $i$th position. If $q_i=0$, then $q^+_i=q^-_i = z_i$ where, by choice of $z$, $c_i\neq z_i$, so the contribution to both expressions is $0$. If $q_i>0$, then the contribution from the $i$th position to $\hat{b}_{c'}(q)$ is $1$ if $c'_i = \varphi^{-1}(c)_i = q_i$, and $0$ otherwise. Similarly, the contribution to $b_c(q^+)-b_c(q^-)$ is $1$ if $c_i = q^+_i = f^{(q_i)}_i = \varphi(q)_i$, and $0$ otherwise. One immediately checks because $\varphi$ is a bijection at any position that the two conditions are equivalent. This works analogously if $q_i<0$.

As the answers given to the signed permutation Mastermind solution are consistent with $\hat{b}_{c'}(q)$, the solution will terminate by outputting the string $c'$. We can consequently compute the codeword $c$ to the original game according to $c=\varphi(c')$.

With high probability, this process uses only $\OO(n)$ queries to obtain $z, f^{(1)}, \dots, f^{(n)}$. Additionally, it uses two times the number of queries used by the signed permutation Mastermind strategy in order to solve the corresponding signed permutation Mastermind instance, which then obtains the codeword for the original game.
\end{proof}

\section{Proof of Theorem \ref{thm:main}}
\label{sec:main}

\begin{figure}[t]
\center
\includegraphics[width = 0.8\textwidth]{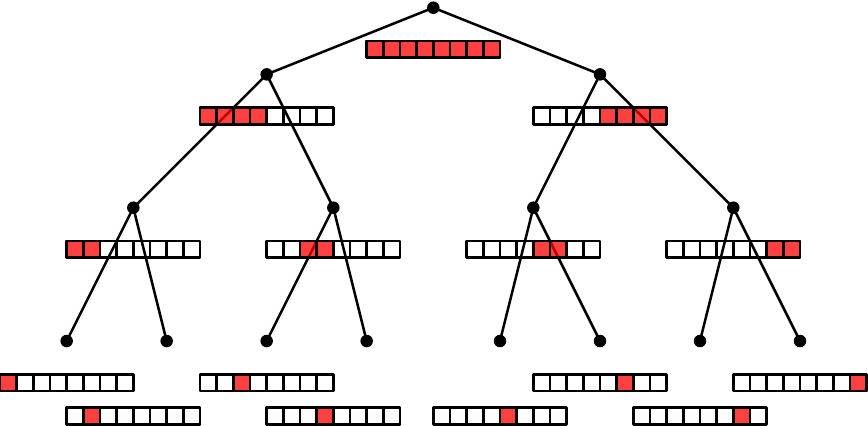}
\caption{Information Tree for $n= 8$. Every node corresponds to an interval, here marked in red.}
\label{fig:informationtree}
\end{figure}

With the reduction from the previous section at hand, we may assume that each color appears in the hidden codeword exactly once. It remains to show that we can determine the position of every color by using $\OO(n)$ signed queries. Before presenting our algorithm, we will first present a token sliding game that will be used to housekeep, at any point while playing signed permutation Mastermind, the information we currently have for each color.

\begin{definition}
Given an instance of signed permutation Mastermind, we define the corresponding information tree $T$ as a rooted complete balanced binary tree of depth $\lceil \log_2 n \rceil$. We denote by $n_T$ the number of leaves of the tree. In other words, $n_T$ is the smallest power of two bigger than or equal to $n$. For each vertex in $T$, we associate a (sub-)interval of $[n_T]$ as follows. We order the vertices at depth $d$ in the canonical way from left to right and associate the $j$th such vertex with the interval $[(n_T/2^d )(j-1)+1, (n_T/2^d )j]$. 
\end{definition}

Note that if $n$ is not a power of two, some vertices will be associated with intervals that go outside $[n]$. An example of an information tree for $n=8$ is illustrated in Figure~\ref{fig:informationtree}.

We introduce handy notation for the complete binary tree $T$. The root of $T$ is denoted by $r$. For any vertex, we denote by $v_L$ and $v_R$ its left and right child respectively if they exist and if we descend multiple vertices we write $v_{LR}$ for $(v_L)_R$. Further $T_L$ denotes the induced subtree rooted at $r_L$, similarly $T_\star$ is the the induced subtree rooted at $r_\star$ for $\star$ being any combination of $R$ and $L$ such that $r_\star$ exists.

For any instance of signed permutation Mastermind, we perform the following \emph{token game} on the information tree as follows. We initially place $n$ colored tokens at the root $r$, one for each possible color in our Mastermind instance. At any point, we may take a token at position $v$ and slide it to either $v_L$ or $v_R$, if we can prove that the position of its color in the hidden codeword lies in the corresponding sub-interval. See Figure \ref{fig:tree01} for an example.%%

\begin{observation} \label{obs:finishleaf}
When all tokens are positioned on leaves of $T$, we know the complete codeword.
\end{observation}

\begin{figure}[t]
\center
\includegraphics[width = 0.7\textwidth]{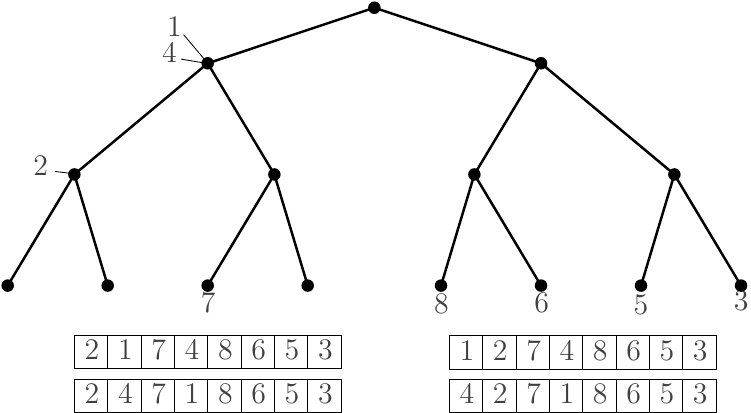}
\caption{Example configuration for $n= 8$. Tokens of colors $1$ through $8$ are at different positions in the information tree and there are 4 remaining possibilities for the codeword.}
\label{fig:tree01}
\end{figure}

The simplest way to move a token is by performing a query that equals that color on, say, the left half of its current position, and zero everywhere else. We call this step querying a token.

\begin{definition}\label{def:weighing}
For a color $f$ with its token at non-leaf node $v$, we say we query the color $f$ if we make a query of the color $f$ only in the left half of the interval corresponding to $v$ (zero everywhere else).
\end{definition}

We note that any query of this form can only give $0$ or $1$ as output. We will refer to any such queries as \emph{zero-one queries}.

\subsection{Solving Signed Permutation Mastermind}

We are now ready to present our main strategy to solve signed permutation Mastermind by constructing a sequence of entropy dense queries that allows us to slide all tokens on $T$ from the root to their respective leaves. This will be done in two steps, which we call \texttt{Preprocess($T$)} and \texttt{Solve($T$)}. 

 As intervals corresponding to vertices close to the root of $T$ are so large, there is initially not much room to include many colors in the same query. Thus the idea of the preprocessing step is to perform $\OO(n)$ \emph{zero-one queries}, in order to move some of the tokens down the left side of the tree.

\begin{algorithm}[h]
\caption{\texttt{Preprocess($T$) }} \label{alg:pre}
\KwIn{Tree $T$}
\KwOut{Preprocessed tree T}
\If{$n_T \le 2$}{Use $1$ query to move all tokens to leafs of $T$\; \Return T\;}
\For{\text{\bf{each}} token $t$ at $r$}{Query $t$ and slide token to either $r_L$ or $r_R$\;
}
\For{\text{\bf{each}} token $t$ at $r_L$}{Query $t$ and slide token to either $r_{LL}$ or $r_{LR}$\;
}
Run \texttt{Preprocess($T_{LL}$)}\;
Run \texttt{Preprocess($T_{LR}$)}\;
\Return {T}\;
\end{algorithm}

\texttt{Preprocess($T$)}, see Algorithm \ref{alg:pre}, takes the tree, queries (Definition~\ref{def:weighing}) all colors whose tokens are at the root $r$ and slides the tokens accordingly. Then repeats for the left child of the root $r_L$. Then we recursively apply the algorithm to the subtrees of the left two grandchildren of the root $T_{LL}$ and $T_{LR}$. If the tree has a depth of $2$ or less, we skip all the steps on vertices that do not exist.

\begin{proposition} \label{prop:preprocess}
The Algorithm~\ref{alg:pre} \texttt{Preprocess($T$)} requires at most $3n_T$ zero-one queries and runs in polynomial time.
\end{proposition}
\begin{proof}
Clearly, if the depth of the tree is $\le 2$ this holds, as we make only a single zero-one query. Then the rest follows by induction, if we analyse the number of queries needed we see, at the root $r$ we need to query at most $n_T$ tokens, and at the left child $r_L$ we query at most $n_T/2$. In the left grandchildren, we recur. So if $a(n_T)$ is the total number of queries we need for \texttt{Preprocess($T$)} for a tree with $n_T$ leafs, then it holds that
\[a(n_T) \le n_T + n_T/2 + a(n_{T_{LL}})+ a(n_{T_{LR}}) = n_T + n_T/2 + a(n_{T}/4)+ a(n_{T}/4)\]
From which follows that $a(n_T) \le 3n_T$. Since we only query tokens all queries we do are zero-one queries and this can be done in polynomial time concluding the proof.
\end{proof}

\begin{figure}[t]
\center
\includegraphics[width = \textwidth]{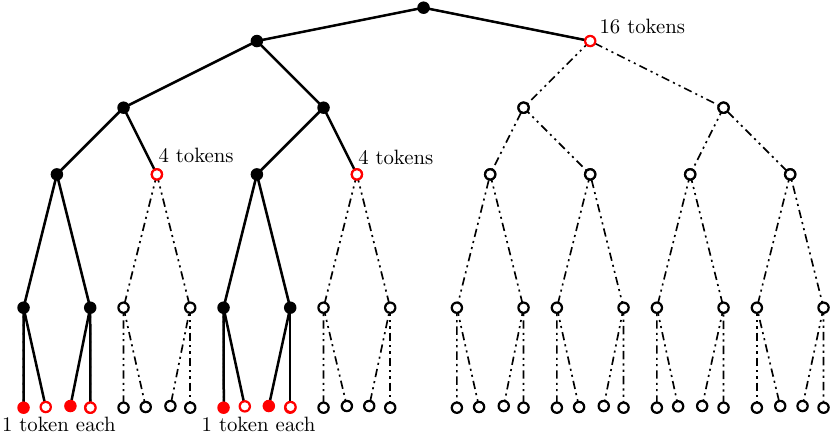}
\caption{Tree preprocessed for n=32, all black vertices have been emptied, tokens are at red vertices.}
\label{fig:pretree}
\end{figure}

The result of running \texttt{Preprocess($T$)} is illustrated in Figure \ref{fig:pretree}. All tokens have either been moved to leaves, or to a vertex of the form $r_R, r_{*R}, r_{**R}, \dots$ where each `$*$' denotes either $LL$ or $LR$. This we call a preprocessed tree. Another way of viewing this is that a tree $T$ is preprocessed if there are no tokens at $r$ or $r_L$, and the subtrees $T_{LL}$ and $T_{LR}$ are preprocessed.~\\ 

Once $T$ is preprocessed we can run the second algorithm \texttt{Solve($T$)}, see Algorithm~\ref{alg:solve}. This is constructed recursively as follows. First note that when $n_T=1$ or $2$, then the preprocessing has already put all tokens at leaves, so the problem is already solved. Now suppose we already know how to run \texttt{Solve($T'$)} for all preprocessed trees $T'$ with $n_{T'}<n$, and let $T$ be a preprocessed tree with $n_T=n$. Observe that preprocessing $T$ means that $T_{LL}$ and $T_{LR}$ are both preprocessed. Hence, we already know procedures \texttt{Solve($T_{LL}$)} and \texttt{Solve($T_{LR}$)} that move all tokens in the respective subtrees to leaves by making a sequence of queries whose support are restricted to the first and second quarter of the available positions respectively. Similarly, as the preprocessing has determined all colors that belong in the right half of the codeword, we know how to move all tokens in $T_R$ to leaves by first performing \texttt{Preprocess($T_R$)} and then \texttt{Solve($T_R$)}, both of which make queries whose support are restricted to the second half of the available positions.

In order to achieve the information-theoretic speedup of queries, the idea is to perform the queries of  \texttt{Solve($T_{LL}$)}, \texttt{Solve($T_{LR}$)} and \texttt{Preprocess($T_R$)} in parallel. This can be thought of as follows. Codebreaker runs the respective algorithms until each of them attempts to make a query. Instead of actually asking codemaker about these queries, codebreaker simply notes the query and pauses the respective algorithm. This continues until all three have proposed queries. Suppose the queries given are $q^{(1)}, q^{(2)}$ and $s$ respectively. Codebreaker combines these into two queries, determined by Lemma \ref{lem:comb} explained later, that are asked to codemaker, and given the answers, codebreaker computes $\hat{b}(q^{(1)})$, $\hat{b}(q^{(2)})$ and $\hat{b}(s)$ that are then presented to the respective algorithms \emph{as if} they were given from codemaker as answers to the respective queries, after which the algorithms are allowed to continue. This is repeated until all three algorithms have terminated. (In case one of the algorithms terminates while another still makes queries, we can simply treat the terminated algorithm as if it is making the all zeros query until the other ones finish.) Finally, this leaves $T_{LL}$ and $T_{LR}$ solved and $T_R$ preprocessed. Hence we can solve $T$ by running \texttt{Solve($T_R$)} (without any parallelism).

In order to combine queries as mentioned above, we employ an idea similar to the Cantor-Mills construction \eqref{eq:cantormills} for coin-weighing.

\begin{lemma}\label{lem:comb}
Define the \emph{support} of a query $q$ 
as the set of indices $i\in [n]$ such that $q_i\neq 0$. For any three queries $q^{(1)}, q^{(2)}$, and $s$ with disjoint supports and such that $s$ is a zero-one query, we can determine $\hat{b}(q^{(1)}), \hat{b}(q^{(2)}),$ and $\hat{b}(s)$ by making only two queries.
\end{lemma}

\begin{proof}
We query $w^{(1)} = q^{(1)} + q^{(2)} + s $ and $w^{(2)} = q^{(1)} - q^{(2)}$, where $+$ and $-$ denote element-wise addition subtraction respectively. Then we can retrieve the answers from just the information of $\hat{b}(w^{(1)})$ and $\hat{b}(w^{(2)})$. If we combine the queries, $\hat{b}(w^{(2)}) + \hat{b}(w^{(2)})= 2 \hat{b}(q^{(1)}) + \hat{b}(s) $. So we can retrieve $\hat{b}(s) = \hat{b}(w^{(2)}) + \hat{b}(w^{(2)}) \mod 2$. And then also the queries, $\hat{b}(q^{(2)}) = ( \hat{b}(w^{(1)}) + \hat{b}(w^{(2)}) - \hat{b}(s)) / 2 $ and $\hat{b}(q^{(2)}) = ( \hat{b}(w^{(1)}) - \hat{b}(w^{(2)}) - \hat{b}(s)) / 2 $ are recoverable.
\end{proof}

\begin{algorithm}[t]
\caption{\texttt{Solve($T$)}} \label{alg:solve}
\KwIn{Preprocessed Tree $T$}
\KwOut{Tree $T$ where all tokens have been queried down to the leaves}
\If{$n_T \le 2$}{\Return T\;}
Run the algorithms \texttt{Solve($T_{LL}$)}, \texttt{Solve($T_{LR}$)} and \texttt{Preprocess($T_R$)} in parallel, and note every time they try to make a query \\
\While{at least one algorithm still has queries}{
Get the next queries $q^{(1)}, q^{(2)}$, and $s$ requested by the algorithms\;
Compute the answers using two queries, as in Lemma \ref{lem:comb}, and return to the respective algorithm\;
}
\texttt{Solve($T_{R}$)}\;
\Return {T}\;
\end{algorithm}

\begin{proposition} \label{prop:solve}
Calling Algorithm~\ref{alg:solve}, \texttt{Solve()}, for a preprocessed tree will move all the tokens to leaves. Moreover, the algorithm will use at most $6n_T$ queries and runs in polynomial time.
\end{proposition}
\begin{proof}
The first statement follows by induction. If $n_T\leq 2$ a preprocessed tree already has all its tokens at leaves, nothing needs to be done. The induction step follows by the correctness of Lemma \ref{lem:comb}.

For the runtime analysis, we also use induction. If the depth $n_T\leq 2$ then clearly the statement holds. If $n_T>2$ the procedure first runs  \texttt{Solve($T_{LL}$)}, \texttt{Solve($T_{LR}$)} and \texttt{Preprocess($T_R$)} in parallel. In each iteration of the main loop, we resolve one query from each of the subprocesses by making two actual queries. This continues until all three processes have terminated, meaning that the number of iterations is the maximum of the number of queries made by the respective subprocesses. After which, we run 
\texttt{Solve($T_{R}$)}. In total, we get the following recursion where $c(n_T)$ is the total number of queries we must make during \texttt{Solve($T$)} and $a(n_T)$ the total number of queries that \texttt{Preprocess($T$)} must make in a tree with $n_T$ leaves.
\begin{align*} 
c(n_T) &\le 2 \cdot \max(c(n_{T_{LL}}) , c(n_{T_{LR}}) , a(n_{T_{R}} )) + c(n_{T_R})\\ &=  2 \cdot \max(c(n_{T}/4) , a(n_{T}/2 )) + c(n_{T}/2)
\end{align*}
By Proposition~\ref{prop:preprocess} we know that $a(n_{T}/2)  \le 3n_T/2$ so by induction we get that $c(n_T) \le 6n_T$. All the operations as well as the computing and decoding of the combined queries in Lemma \ref{lem:comb} are clearly in polynomial time so this concludes the proof.
\end{proof}

Now we have all the tools at hand to prove our main result.

\begin{proof}[Proof of Theorem \ref{thm:main}.]
We have Lemma~\ref{lem:perm} which reduces the problem of solving black-peg Mastermind to solving signed permutation Mastermind. For the new instance of signed permutation Mastermind that results from this, we consider the information tree. We move the tokens of this tree to the leaves by first running the algorithm \texttt{Preprocess($T$)} and then applying the algorithm \texttt{Solve($T$)} on the preprocessed tree. By Propositions \ref{prop:preprocess} and \ref{prop:solve} this takes at most $9n_T$ signed queries and is done in polynomial time. By Observation~\ref{obs:finishleaf} we have found the hidden codeword of the signed permutation Mastermind, and therefore also the hidden codeword of the black-peg Mastermind game. The transformation can be done in polynomial time and by Lemma~\ref{lem:perm} we need at most $\OO(n + 9n_T) = \OO(n)$ queries to solve black-peg Mastermind. 
\end{proof}

%!TEX root = mastermind.tex

\section{Playing Mastermind with arbitrarily many colors} \label{sec:general}

We briefly make some remarks on other ranges of $k$ and $n$ for Mastermind. By combining Theorem \ref{thm:main} with results of Chv\'atal \cite{chvatal1983mastermind} and Doerr, Doerr, Sp\"ohel, and Thomas \cite{DDST16}, we determine up to a constant factor the smallest expected number of queries needed to solve black-peg and black-while-peg Mastermind for any $n$ and $k$.

%\general* 

For any $n$ and $k$, let $\bmm(n, k)$ denote the minimum (over all strategies) worst-case (over all codewords) expected number of guesses to solve Mastermind if only black peg information can be used. Similarly, denote  $\bwmm(n, k)$ the smallest expected number of queries needed if both black and white peg information can be used. The following relation between black-peg and black-white-peg Mastermind was shown by Doerr, Doerr, Sp\"ohel, and Thomas.

\begin{theorem}[{Theorem 4, \cite{DDST16}}] For all $k\geq n \geq 1$,
\[\bwmm(n, k) = \Theta\left( \bmm(n, n) + k/n\right).\]
\label{thm:doerr}
\end{theorem}

\begin{proof}[Proof of Theorem \ref{thm:general}]
For small $k$, say $k\leq \sqrt{n}$, the result follows by the results of Chv\'atal \cite{chvatal1983mastermind}. Moreover, for $k\geq n$ the white peg statement follows directly by combining Theorems \ref{thm:main} and \ref{thm:doerr}. Thus it remains to consider the case of $\sqrt{n}\leq k \leq n$, and the case of $k\geq n$ for black-peg Mastermind.

For $\sqrt{n}\leq k \leq n$, the leading terms in both bounds in Theorem \ref{thm:general} are of order $n$, which matches the entropy lower bound.  On the other hand, using Lemma \ref{lem:zero} we can find a query such that $b(z)=0$ in $\OO(n)$ queries. Having found this, we simply follow the same strategy as for $n$ color black-peg Mastermind by replacing any color $>k$ in a query by the corresponding entry of $z$. Thus finishing in $\OO(n)$ queries.

Finally, for black-peg Mastermind with $k \ge n$, the leading term in Theorem \ref{thm:general} is of order $k$. This can be attained by using $k$ guesses to determine which colors appear in the codeword and then reduce to the case of $n$ colors. On the other hand, $\Omega(k)$ is clearly necessary as this is the expected number of queries needed to guess the correct color in a single position, provided the codeword is chosen uniformly at random.
\end{proof}

\bibliographystyle{abbrv}
\bibliography{mastermind}

\end{document}